\documentclass[a4paper,leqno,12pt]{amsart}

\usepackage{epsf,graphicx,epsfig}
\usepackage{amscd}
\usepackage{amsmath,latexsym,amssymb,amsthm}
\usepackage[all]{xy}
\usepackage[nospace,noadjust]{cite}
\usepackage{setspace,cite}
\usepackage{lscape,fancyhdr,fancybox}
\usepackage{graphicx}

\usepackage{color}
\usepackage{url}
\usepackage{enumerate}
\usepackage[mathscr]{euscript}
\input xy
\xyoption{all}

  \theoremstyle{plain}

    \newtheorem{thm}{Theorem}[section]
    \newtheorem{prop}[thm]{Proposition}
   \newtheorem{lemma}[thm]{Lemma}

    \newtheorem{subsec}[thm]{}

\theoremstyle{definition}

\theoremstyle{remark}
     \newtheorem{remark}[thm]{Remark}
    \newtheorem*{ack}{Acknowledgements}

\newcommand{\Z}{\mathbb{Z}}
\newcommand{\R}{\mathbb{R}}
\newcommand{\Q}{\mathbb{Q}}
\newcommand{\C}{\mathbb{C}}
\newcommand{\orgr}{\widetilde{I}_{2n,k}}
\newcommand{\omgr}{\widetilde{I}_{2m,l}}

\title{}
\author{}
\date{}
\usepackage{amssymb}

\begin{document}
\title{Degrees of Maps between Isotropic Grassmann Manifolds}
\author{Samik Basu}
\email{samik.basu2@gmail.com; samik@rkmvu.ac.in}
\address{Department of Mathematics,
 Vivekananda University,
 Belur, Howrah 711202,
West Bengal, India.}

\author{Swagata Sarkar}
\email{swagatasar@gmail.com}
\address{Department of Mathematics, 
 UM-DAE Centre for Excellence in Basic Sciences, 
University of Mumbai, Vidyanagari Campus, Kalina,
Mumbai - 400098, India.}

\date{\today}

\thispagestyle{empty}

\begin{abstract}

Let $\orgr$ denote the space of $k$-dimensional, oriented isotropic subspaces of $\R^{2n}$, called the oriented isotropic Grassmannian. Let $f \colon \orgr \rightarrow \omgr $ be a map between two oriented isotropic Grassmannians of the same dimension, where $k,l \geq 2$. We show that either  $(n,k) = (m,l)$ or $\deg{f} = 0$. Let $\R\widetilde{G}_{m,l}$ denote the oriented real Grassmann manifold. For $k,l \geq 2$ and $\dim{\orgr} = \dim{\R\widetilde{G}_{m,l}}$, we also show that the degree of  maps $g \colon \R \widetilde{G}_{m,l} \rightarrow \orgr $ and $h \colon \orgr \rightarrow \R \widetilde{G}_{m,l}$ must be zero.
\end{abstract}

\maketitle

\section{Introduction}
It has been proved in \cite{rs} that maps between two different oriented real Grassmann manifolds of the same dimension cannot have non-zero degree, provided the target space is not a sphere. A similar result is obtained for complex Grassmann manifolds in \cite{psr}, when the map is a morphism of projective varieties. For arbitrary maps, this result has been verified for the complex Grassmann manifolds for many cases in \cite{rs} and \cite{ss}.\\

In this paper we consider the analogous question for the space $\orgr$ of oriented $k$-dimensional isotropic subspaces of a symplectic vector space of dimension $2n$. The oriented isotropic Grassmannian was considered in \cite{mor-nig} and its cohomology was computed with real coefficients. Their method involves identifying $\orgr$ as a homogeneous space $\orgr \simeq U(n)/(SO(k)\times U(n-k))$. One may similarly consider $I_{2n,k}$, the isotropic Grassmannian of $k$-dimensional isotropic subspaces of a symplectic $2n$ dimensional vector space, which is $\simeq U(n)/(O(k)\times U(n-k))$. It turns out that the isotropic Grassmannian is orientable if and only if $k$ is odd (\cite{mik}). In this paper we consider maps between oriented isotropic Grassmannians of the same dimension and prove 
\begin{thm}
Let $n,k,m,l$ be integers such that $2 \leq l \leq m $  and $\dim{\orgr} = \dim{\omgr}$. Let $f \colon \orgr \rightarrow \omgr $. Then either $(n,k) = (m,l)$ or $\deg{f} = 0$.
\end{thm}
(see Theorem \ref{main1}). Note that $\widetilde{I}_{2n,1} \simeq S^{2n-1}$ and so it is possible to get maps of arbitrary, non-zero degree, $\phi: \orgr \to \widetilde{I}_{2m,1}$ whenever $\dim(\orgr)= 2m-1$. We also prove
\begin{thm}
Consider maps $h \colon \orgr \rightarrow \R \widetilde{G}_{m, l}$ and $g \colon \R \widetilde{G}_{m,l} \rightarrow \orgr $, where $2 \leq l \leq m $, $2 \leq k \leq n $ and $\dim{\orgr} = \dim{\R\widetilde{G}_{m,l}}$. Then $\deg{g} = \deg{h} = 0 $.
\end{thm}
The main technique used to prove the statements above is the result that if $f:X\to Y$ (with $\dim X = \dim Y$) is a map of non-zero degree, then $f^*:H^*(Y;\Q)\to H^*(X;\Q)$ is a monomorphism. We obtain some results on the structure of the cohomology ring of $\orgr$ to deduce the above theorems. \\

The paper is organised as follows. In section 2 we recall the description of the spaces $I_{2n,k}$ and $\orgr$ and express them as  homogeneous spaces. In section 3 we compute the cohomology of  $\orgr$. In section 4 we prove the main theorems.

\section{Isotropic Grassmannian}

In this section we set up the relevant notation and describe the isotropic Grassmannian as a homogeneous space. For $\mathbb{F}$ = $\mathbb{R}$ or $\mathbb{C}$ and $n \in \mathbb{N}$, let $\mathbb{F}^{n}$ denote the $n$-dimensional 
$\mathbb{F}$-vector space (upto isomorphism). Further, let $M(n, \mathbb{F})$ denote the group of linear maps $\mathbb{F}^{n} \rightarrow \mathbb{F}^{n}$,
and let $GL(n,\mathbb{F})$ denote the group of automorphisms $\mathbb{F}^{n} \rightarrow \mathbb{F}^{n}$. Let $U(n)$ := $U (n ; \mathbb{C})$ denote the group of 
unitary linear transformations $\mathbb{C}^{n} \rightarrow \mathbb{C}^{n}$, and $O(n)$ := $O(n ; \mathbb{R})$ denote the group of orthogonal linear transformations
$\mathbb{R}^{n} \rightarrow \mathbb{R}^{n}$. \\

Choose a symplectic form $\omega$ on $\mathbb{R}^{2n}$. Let $Sp(n)$ := $Sp(n ; \mathbb{R})$ denote the set of linear transformations $\mathbb{R}^{2n} \rightarrow \mathbb{R}^{2n}$ 
which preserve this symplectic form $\omega$. 
Choose a basis $\mathfrak{B}$ = $(e_{1} , \cdots , e_{n} ; f_{1} , \cdots , f_{n} )$, of $\mathbb{R}^{2n}$, such that with respect to this basis, $\omega$ 
can be written as :
$$ \omega = de_{1} \wedge df_{1} + \cdots + de_{n} \wedge df_{n}$$

We coordinatise $\mathbb{R}^{2n}$ with respect to this basis and identify $\mathbb{R}^{2n}$ with $\mathbb{C}^{n}$ via the following map: 
$${\bf r} \colon \mathbb{R}^{2n} \longrightarrow \mathbb{C}^{n} $$
$$(x_{1}, \cdots , x_{n} ; y_{1} , \cdots , y_{n}) \mapsto (x_{1} + \imath y_{1} , \cdots ,x_{n} + \imath y_{n} ) $$

Then, ${\bf r}$ induces a map $M(n; \mathbb{C}) \longrightarrow M(2n; \mathbb{R})$. By abuse of notation, we will also call this map ${\bf r}$.
Let $Sp(n)$ := $Sp(n; \mathbb{R}) \in M(2n ; \mathbb{R})$ denote the set of isomorphisms which preserve the symplectic forms $\omega$. Then the image of $U(n)$
under ${\bf r}$ lies in $Sp(n)$. Hence the usual action of $U(n)$ on $\mathbb{C}^{r}$ induces an action of $U(n)$ on $\mathbb{R}^{2n}$
via symplectic morphisms. \\

Define the isotropic Grassmannian, $I_{2n,k}$, to be the space of $k$-dimensional isotropic vector subspaces of $\mathbb{R}^{2n}$. Then one has the following proposition.\\

  \begin{prop}\label{prop21}
   The isotropic Grassmannian, $I_{2n,k}$, is diffeomorphic to the quotient ${\displaystyle {U(n)}/{(O(k) \times U(n-k))}}$.
  \end{prop}

\begin{proof}
Note that $U(n)$ acts on $I_{2n,k}$. Consider any $k$-dimensional isotropic subspace, 
$V \subset \mathbb{R}^{2n}$. Via the identification of $\mathbb{R}^{2n}$ with $\mathbb{C}^{n}$, 
the complement, denoted by $W$ say, of $V \oplus \imath V$ with respect to the form $\omega$,  is a complex 
subspace of $\mathbb{C}^{n}$. Moreover, $(V \oplus \imath V ) \oplus W \cong \mathbb{C}^{n}$ is a 
decomposition with respect to the standard Hermitian inner product on $\mathbb{C}^{n}$.\\
 
Now, given any two $k$-dimensional, isotropic subspaces $V$ and $V'$, one has identifications :
 $$(V \oplus \imath V ) \oplus W \cong \mathbb{C}^{n}$$
 $$(V' \oplus \imath V' ) \oplus W' \cong \mathbb{C}^{n}$$
Therefore, one can choose an orthogonal transformation $\varphi \colon V \rightarrow V'$, which takes $V$ to $V'$ and an isometry 
$\psi$ which takes $W$ to $W'$. Since $\varphi$ is orthogonal, $\imath \varphi$ will take $\imath V$ to $\imath V'$. Thus one obtains
an isometry $\mathbb{C}^{n}$ to $\mathbb{C}^{n}$ which takes $V$ to $V'$, and hence, the action of $U(n)$ is transitive. \\
 
 Let $V_{k}$  denote the $k$-dimensional subspace of $\mathbb{R}^{2n}$ generated by the basis vectors 
 $e_{1} , \cdots , e_{k} $. 
 Then any isometry $A \in U(n)$ with $A(V_{k})$ = $V_{k}$, is orthogonal when restricted to $V_{k}$. Additionally, 
 it gives isomorphisms  $V_{k} \oplus \imath V_{k} \rightarrow V_{k} \oplus \imath V_{k} $ and 
 $(V_{k} \oplus \imath V_{k})^{\perp} \rightarrow  (V_{k} \oplus \imath V_{k})^{\perp} $. 
 Therefore, the stabilizer of $V_{k}$ is $O(k) \times U(n-k)$.
 
 \end{proof}
 
It follows that $I_{2n, k}$ is a $2k(n-k) + [{\displaystyle{k(k+1)}/{2}}]$-dimensional manifold. In fact, for $k$ = $1$, $I_{2n, k}$
is the real projective space, $\mathbb{RP}^{2n-1}$. One notes that $I_{2n,k}$ is orientable if and only if  $k$  is odd (\cite{mik}). \\

We consider  $\widetilde{I}_{2n,k}$, the space of  $k$-dimensional, oriented, isotropic subspaces of $\mathbb{R}^{2n}$, called the oriented isotropic Grassmannian.  
The oriented isotropic Grassmannian, $\widetilde{I}_{2n,k}$, is again a $2k(n-k) + [{\displaystyle{k(k+1)}/{2}}]$-dimensional manifold. As in Proposition \ref{prop21}, one has:
 
  \begin{prop}
   The isotropic Grassmannian, $\widetilde{I}_{2n,k}$, is diffeomorphic to the quotient ${\displaystyle {U(n)}/{(SO(k) \times U(n-k))}}$.
  \end{prop}

\section{Cohomology of the Oriented Isotropic Grassmannian}

In this section we compute the cohomology of the oriented isotropic Grassmannian, $\widetilde{I}_{2n,k}$. The cohomology with $\R$ coefficients was computed in \cite{mor-nig} using formulas for the real cohomology of homogeneous spaces. We compute the same algebraically, fixing appropriate notation along the way. We use the Serre spectral sequence and the following fibrations:
$$
 \widetilde{I}_{2n,k} \longrightarrow BSO(k) \times BU(n-k) \longrightarrow BU(n)
$$
$$
 U(n) \longrightarrow \widetilde{I}_{2n,k} \longrightarrow BSO(k) \times BU(n-k)
$$

The first fibration induces the Serre spectral sequence with $E_{2}^{p,q}$ term given by
\begin{equation}
 \label{sss1}
 E_{2}^{p,q} = H^{p} (BU(n); \mathbb{Q}) \otimes H^{q} (\widetilde{I}_{2n,k} ; \mathbb{Q})
\end{equation}
which converges to $H^{p+q} (BSO(k) \times BU(n-k) ; \mathbb{Q})$.\\

The second fibration induces the Serre spectral sequence with $E_{2}^{p,q}$ term given by
\begin{equation} \label{sss2}
 E_{2}^{p,q} = H^{q} (U(n); \mathbb{Q}) \otimes H^{p} (BSO(k) \otimes BU(n-k) ; \mathbb{Q})
\end{equation}
which converges to $H^{p+q} (\widetilde{I}_{2n,k} ; \mathbb{Q})$.\\

It is well-known (\cite{mt}) that 
 $$H^{*} (BU(n-k); \mathbb{Q}) \cong \mathbb{Q} [c_{1} , c_{2} , \cdots , c_{n-k}]$$
where $c_{i} \in H^{2i}(BU(n-k);\mathbb{Q} )$ is the $i^{th}$ Chern class of the universal complex $(n - k)$-plane bundle $\gamma_{n-k}$; and, 
$$ H^{*} (U(n); \mathbb{Q}) \cong \wedge_{\mathbb{Q}} [x_{1} , x_{3} , \cdots , x_{2n-1}]$$
where $x_{i} \in H^{i}(U(n); \mathbb{Q})$ and $\wedge$ denotes the exterior algebra. \\

For odd $k$ (= $2m + 1$)
 $$H^{*} (BSO(k); \mathbb{Q}) \cong \mathbb{Q} [p_{1} , p_{2} , \cdots , p_{m}]$$
where $p_{i} \in H^{4i}(BSO(k); \mathbb{Q})$ are the Pontrjagin classes of the universal oriented $k$-plane bundle $\xi_k$. 
For the case $k$ = $2m$
 $$ H^{*} (BSO(k); \mathbb{Q}) \cong \mathbb{Q} [p_{1} , p_{2} , \cdots , p_{m-1} , e_k ]$$
where $e_k \in H^{k}(BSO(k); \mathbb{Q})$ is the Euler class of $\xi_k$. In this case one has $p_m(\xi_k)=e_k^2$.\\

Note that the inclusion $SO(k)\times U(n-k) \subset U(n)$ is induced by $(\R^k\otimes \C) \oplus \C^{n-k}\cong \C^n$. 
 It follows that on classifying spaces $BSO(k)\times BU(n-k) \to BU(n)$ classifies the complex bundle $\xi_k\otimes \C \oplus \gamma_{n-k}$. 
 Hence we have a commutative diagram of fibrations 
 $$\xymatrix{ U(n) \ar[r] \ar@{=}[d] & \orgr \ar[r] \ar[d] & BSO(k)\times BU(n-k) \ar[d]^{(\xi_k\otimes \C)\oplus \gamma_{n-k}}\\ 
              U(n) \ar[r] & EU(n) \ar[r] & BU(n) }$$
This induces a diagram  of spectral sequences which we may use to compute the differentials in \ref{sss2}. Let $\lambda^*$ denote the homomorphism from the spectral sequence $\wedge_\Q(x_1,x_3\cdots) \otimes \Q[c_1,c_2,\cdots] \implies H^*(\mathit{pt};\Q)$ to \ref{sss2}. As the classes $x_{2i-1}$ are transgressive with $d_{2i}(x_{2i-1})=c_i$, so are the classes $\lambda^\ast(x_{2i-1})$. Therefore we obtain 
\begin{equation}\label{diff}
 \begin{array}{lll}
    d(x_{2i -1})
    &=& \lambda^*(c_i) \\
    &=& c_{i} ((\xi_{k} \otimes \mathbb{C}) \oplus \gamma_{n-k})\\
    &=& \sum_{j=0}^{\infty} c_{j} (\xi_{k} \otimes \mathbb{C}) c_{i-j} (\gamma_{n-k})\\
    &=& \sum_{j=0}^{[i/2]} p_{j} c_{i-2j}\\
    &=& c_i + \sum_{j=1}^{[i/2]} p_{j} c_{i-2j}\\
   \end{array}
\end{equation}

\begin{prop}\label{initial}
Let $2 < k \leq n $. If $k < n$, the cohomology groups $H^i(\widetilde{I}_{2n,k})$ are $0$ if $i\leq 3$ and $H^4(\orgr)$ is generated by $p_1$. In the case $k=n$,  the cohomology group $H^{1}(\widetilde{I}_{2n,n})$ is isomorphic to $\Z$ and $H^{4}(\widetilde{I}_{2n,n})$ is zero.
\end{prop}

\begin{proof}
 For $k=n$, the space $\orgr \cong U(n)/SO(n)$, thus the fundamental group and hence $H^1$ is $\cong \Z$. 
 Otherwise in the spectral sequence $\ref{sss2}$ one has a class $c_1$ in $E_2^{2,0}$. In degrees $\leq 3$ the spectral sequence \ref{sss2} is $\cong$ $\wedge(x_1,x_3)\otimes \Q[c_1]$ and from \ref{diff} we get that $d_2(x_1)=c_1$. Hence the only possible class in $H^{*\leq 3}$ is $x_3$. 

Note that \ref{diff} also gives $d_4(x_3)= c_2 + p_1$ if $k\geq 2$ and $d_4(x_3)= c_2$ if $k=1$. Thus we conclude that $H^{*\leq 3}$ is $0$ if $k<n$, and if in addition $n>k\geq 2$ then $H^4(\orgr) (\cong \Q)$ generated by $p_1$.  In the case $k=n$, $d_4(x_3)= p_1$, and hence $H^4(\orgr)$ becomes zero.
\end{proof}

We may compute further in the spectral sequence $\ref{sss2}$. Notice that the formula \ref{diff} is of the form $d(x_{2i-1}) = c_i + \cdots$ and so the class $c_i$ is not zero if $i\leq n-k$. Thus the elements $d(x_1), d(x_3) \ldots d(x_{2(n-k)-1})$ form a regular sequence in $E_2^{*,0}$. It follows that no multiple of $x_{2j -1}$, for $j\leq n-k$, can be a permanent cycle. Therefore any positive degree classes surviving to the $E_\infty$-page must have degree $> 2(n-k)+1$. In fact we have the Proposition 

\begin{prop}
Suppose $2 \leq k <n$. The cohomology algebra $H^\ast(\orgr;\Q)$ has algebra generators $p_1,\cdots , p_m$ in degrees $4,8,\cdots$ when $k=2m+1$ is odd. If $k=2m$ there is an additional generator $e_m$ in degree $2m$ that satisfies $e_m^2 = p_m$. Other algebra generators are in degrees $\ge 2(n-k)+1$.   

\end{prop}

\begin{proof}
In view of the discussion above it suffices to prove the first two statements for the horizontal $0$-line $E_\infty^{\ast,0}$. Note that for $j$ odd, the equation \ref{diff} gives 
$$d(x_{2j-1}) =  c_j + \sum_{l=1}^{[j/2]} p_{l} c_{j-2l}$$

Inductively we conclude that $d_{2j}(x_{2j-1})= c_j$. We have $d_2(x_1)=c_1$ and thus $c_1$ is $0$ in the $E_3$-page. Inductively $c_{j-1}$ is $0$ in the group $E_{2j-3}^{2j-4,0}$. Hence, the equation above implies $d_{2j}(x_{2j-1})=c_j$ as the other $c_{odd\leq j-1}$ are $0$ in $E_{2j}$. It follows that $c_j$ is $0$ in $E_{2j+1}$. 

The remaining classes in the horizontal $0$-line are $p_i$ , $c_{2j}$ and $e_m$ if $k$ is even. The remaining differentials are generated by 
$$d(x_{4j-1}) =  \sum_{l=0}^{j} p_{l} c_{2j-2l}$$
Hence the horizontal $0$-line is the graded algebra $A(n,k)$ below. The Proposition now follows from Lemma $\ref{orgr}$ and \cite{rs}.
\end{proof}

Let $\R \widetilde{G}_{n,k}$ denote the oriented real Grassmannian of all $k$-dimensional oriented subspaces of $\mathbb{R}^{n}$. As a space this is $\simeq SO(n)/SO(k)\times SO(n-k)$. Let $\C G_{n,k}$ denote the Grassmannian of $k$-planes in $\C^n$.  
Define the graded algebras $A(n,k)$ for $n>k$ as (with notations as above) 
$$
A(n,k) = \left\{ \begin{array}{rl}
\frac{\Q[p_1,\cdots, p_m, c_2,c_4,\cdots, c_{n-2m-2}]}{(d(x_3),\ldots, d(x_{2n-1}))} &\mbox{ if $n$ is even, $k=2m+1$} \\
 \frac{\Q[p_1,\cdots, p_m, c_2,c_4,\cdots, c_{n-2m-1}]}{(d(x_3),\ldots, d(x_{2n-3}))} &\mbox{ if $n$ is odd, $k=2m+1$} \\
\frac{\Q[p_1,\cdots, p_{m-1},e_m, c_2,c_4, \cdots, c_{n-2m}]}{(d(x_3),\ldots, d(x_{2n-1}))} &\mbox{ if $n$ is even, $k=2m$}\\
\frac{\Q[p_1,\cdots, p_{m-1},e_m, c_2,c_4, \cdots, c_{n-2m-2}]}{(d(x_3),\ldots, d(x_{2n-3}))} &\mbox{ if $n$ is odd, $k=2m$}
\end{array} \right.
$$

\begin{lemma}\label{orgr} There are isomorphisms of graded algebras \\
a) $A(2s,2m+1) \cong H^{*/2}(\C G_{s-1,m};\Q)$\\
b) $A(2s+1,2m+1) \cong H^{*/2}(\C G_{s,m};\Q)$\\
c) $A(2s,2m) \cong H^*( \R \widetilde{G}_{2s+1,2m};\Q)$\\
d) $A(2s+1,2m) \cong H^*(\R \widetilde{G}_{2s+1,2m};\Q)$\\
\end{lemma}

\begin{proof}
This follows from the computation of the cohomology algebras of the Grassmannians using (with $\Q$ coefficients) : 
$$H^*(\C G_{n,k}) \cong H^*(BU(k))\otimes H^*(BU(n-k))/(H^+(BU(n))$$
and if $n$ is odd.
$$H^*(\R \widetilde{G}_{n,k}) \cong H^*(BSO(k))\otimes H^*(BSO(n-k))/(H^+(BSO(n))$$
\end{proof}

\begin{remark}
One may compare this to the expression obtained in \cite{mor-nig}. Observe that the ring of characteristic classes $\mathcal{A}$ of the principal bundle $U(n)\to \orgr$ matches the graded algebra $A(n,k)$ above. Note that the cohomology of $\orgr$ is $\mathcal{A}\otimes \Lambda$ where $\Lambda$ is an exterior algebra on classes in degrees $d\in S_{n,k}$ (\cite{mor-nig}, Theorem 1.7) with  
$$
S_{n,k}= \left\{ \begin{array}{rl}
\{4[\frac{n-k+1}{2}]+1, 4[\frac{n-k+1}{2}]+3,\cdots, 2n-3\} &\mbox{ if both $n$ and $k$ are even} \\
\{4[\frac{n-k+1}{2}]+1, 4[\frac{n-k+1}{2}]+3,\cdots, 2n-1\} &\mbox{otherwise}
\end{array} \right.
$$
 It follows that the Poincar\'e polynomial of $\orgr$ is given by 
$$p_{\orgr} (x) = p_{A(n,k)}(x) \Pi_{d\in S_{n,k}} (1+x^d)$$
which may be computed from the known formulas for complex and real Grassmannians.  
\end{remark}

Recall that height of a nilpotent element, $x$, in an algebra, is defined to be the least positive integer $n$, such that $x^{n} \neq 0$ but $x^{n+1} = 0$.

\begin{prop}\label{orderp1}
The height of the element, $p_{1}$ in $H^*(\orgr)$ is $t(s-t)$ for $(n,k)\in \{(2s+2,2t+1), (2s+1,2t+1), (2s,2t),(2s+1,2t)\}$ . 
\end{prop}
\begin{proof}
Follows from Lemma \ref{orgr} and Lemma 4 of \cite{rs}. 
\end{proof}

\section{Main results}

In this section we consider the question of possible Brouwer degrees of maps $f \colon \orgr \rightarrow \omgr$ , where $\orgr$ and $\omgr$ are oriented isotropic Grassmannians, such that $\dim{\orgr} = \dim{\omgr}$.

Note that when $l=1$, the space $\omgr \simeq \displaystyle{U(m)}/{U(m-1)} \simeq S^{2m-1}$. Also note that  $\dim{\orgr}$ ($=2k(n-k) + \frac{k(k+1)}{2}$) is odd if and only if $k \equiv 1,2~ (mod~ {4}) $. In these cases ($\dim{\orgr} = 2m-1$)  given any $\lambda \in \Z$, there exists a map $f_{\lambda} \colon \orgr \rightarrow S^{2m-1}$ with $\deg{f_{\lambda}} = \lambda$. 
 We prove that these are the only possible cases of non zero degree.

\begin{thm} \label{main1}
Let $n,k,m,l$ be integers such that  $2 \leq k \leq n $ and $2 \leq l \leq m $  and $\dim{\orgr} = \dim{\omgr}$. Let $f \colon \orgr \rightarrow \omgr $. Then either $(n,k) = (m,l)$ or $\deg{f} = 0$.
\end{thm}
 
\begin{proof}

Suppose $n=k$ and $m=l$. Then $\dim{\orgr} = \dim{\omgr}$ implies $n(n+1)/2 = m(m+1)/2$ and it follows that $n =m$. Note that the space $\widetilde{I}_{4,2} \simeq \displaystyle{U(2)}/{SO(2)}$ is an oriented manifold of dimension $3$. Observe that $k \geq 2$ and $n \neq k$ implies $\dim{\orgr} \geq 4$. Hence $\dim \orgr = \dim \omgr$ implies $(n,k)=(m,l)$ if $n=k,m=l$ or one of $(n,k)$ or $(m,l)$ equals $(2,2)$.   

Consider the case where $n=k$, $n>2$ and $m \neq l$. We use  the  fact that if $\deg{f} \neq 0 $ then $f^*$ is a monomorphism on cohomology with rational coefficients. By Proposition $\ref{initial}$, $H^{4}(\widetilde{I}_{2n,n})=0$ and $H^4(\omgr) \neq 0$. Hence  $\deg{f} =0$. In the case $n \neq k$, $m = l$ and $m>2$, we have, again by Proposition $\ref{initial}$, that $H^1 (\widetilde{I}_{2m,m}) \cong \Z$ and $H^1(\orgr)=0$. Therefore $\deg{f} = 0$. 

Now we proceed to the more general case $2 \leq k < n $ and $2 \leq l < m $. Consider $f^{*} \colon H^{*}(\omgr ; \Q) \rightarrow H^{*}(\orgr ; \Q)$. Since $p_1$ is the generator of $H^{4}(\omgr ; \Q)$ we must have  (denote by $p_{1}(m,l)$ the class $p_1 \in H^{4}(\omgr ; \Q)$) 
$$f^*p_1(m,l) = \lambda p_1(n,k)$$
 By Proposition $\ref{orderp1}$ the height of $p_{1}(n,k)$ is $[k/2][(n-k)/2]$ and the height of $p_{1}(m,l)$  is $[l/2][(m-l)/2]$.  ( Here, $[t]$ denotes the integral part of $t$.)  

If $\deg{f} \neq 0$ then $f^{*} $ is a monomorphism and so $\lambda \neq 0$. Moreover $p_1(m,l)^{[\frac{l}{2}][\frac{m-l}{2}]} \neq 0 $ implies 

$$f^*p_1(m,l)^{[\frac{l}{2}][\frac{m-l}{2}]} = \lambda^{[\frac{l}{2}][\frac{m-l}{2}]} p_1(n,k)^{[\frac{l}{2}][\frac{m-l}{2}]}\neq 0$$
$$\implies \left[ \frac{l}{2} \right] \left[ \frac{m-l}{2} \right] \leq \left[ \frac{k}{2} \right] \left[ \frac{n-k}{2} \right] .$$
 Since $f^{*}$ is a ring homomorphism, we have
 
$$0= f^*p_1(m,l)^{[ \frac{l}{2}] [ \frac{m-l}{2}] +1} = \lambda^{[\frac{l}{2}][\frac{m-l}{2}]+1} p_1(n,k)^{[\frac{l}{2}][\frac{m-l}{2}]+1}$$
$$\implies  \left[ \frac{k}{2} \right] \left[ \frac{n-k}{2} \right] \leq \left[ \frac{l}{2} \right] \left[ \frac{m-l}{2} \right] $$
Therefore $ [\frac{k}{2}][\frac{n-k}{2}] = [\frac{l}{2}][\frac{m-l}{2}]$. Together with the equation $2k(n-k) + \frac{k(k+1)}{2}=2l(m-l) + \frac{l(l+1)}{2}$ we prove that it leads to a contradiction. Assume that $k\leq l$ (there is no loss of generality in doing this.) The above equality implies  $l(m-l) - 4\leq k(n-k) \leq l(m-l)+ 4$. Rearranging terms we obtain $-16 \leq (k-l)(k+l+1) \leq 16$. 

As both $k\geq 2$ and $l\geq 2$ we have $k+l+1 \geq 5$ and so the above inequality can hold only when $k-l=0,1,2$. Observe that $k=l$ implies $n=m$ so that $(n,k)=(m,l)$. Note also that $(k-l)(k+l+1)$ must also be divisible by $4$ being equal to $k(n-k)-l(m-l)$. If $k=l+2$, we have $(k-l)(k+l+1)=2(2l+3)$ is not divisible by $4$. If $k=l+1$ we have $(k-l)(k+l+1) = 2l+2$ which is divisible by $4$ only when $l$ is odd. Therefore the allowed values of $l$ are $3,5,7$. 

{\bf  Case $l=3$ :} We have $k=4$ and the equation $8(n-4) + 10 = 6(m-3) + 6$ which implies $4n = 3m +5$. This implies $m= 4s +1, n = 3s +2$ for some positive integer $s$.  The equation $[\frac{l}{2}][\frac{m-l}{2}] = [\frac{k}{2}][\frac{n-k}{2}]$ implies $2s -1 = 2[\frac{3s - 2}{2}]$. But the LHS is bigger for $s=1$ and the RHS is always bigger for $s>1$. 

{\bf  Case $l=5$ :} We have $k=6$ and the equation $12(n-6) + 21 = 10(m-5) + 15$ which implies $6n = 5m +8$ which has the only solution $m=6s+2, n = 5s+3$. The equation $ [\frac{k}{2}][\frac{n-k}{2}] = [\frac{l}{2}][\frac{m-l}{2}]$ implies $3[\frac{5s-3}{2}] = 2 (3s - 2)$. The LHS is always bigger for $s>0$.

{\bf  Case $l=7$ :} We have $k=8$ and the equation $16(n-8) + 36 = 14(m-7) + 28$ which implies $8n = 7m +11$ which has the only solution $m=8s+3, n = 7s+4$. The equation $ [\frac{k}{2}][\frac{n-k}{2}] = [\frac{l}{2}][\frac{m-l}{2}]$ implies $4[\frac{7s-4}{2}] = 3 (4s - 2)$. For $s=1$ , the LHS is $4$ and the RHS is $6$. For $s>1$ the LHS is bigger.

\end{proof}

\mbox{ }\\

The arguments in the above case can be extended to prove the following:

\begin{thm}
Consider maps $h \colon \orgr \rightarrow \R \widetilde{G}_{m, l}$ and $g \colon \R \widetilde{G}_{m,l} \rightarrow \orgr $, where $2 \leq l \leq m $, $2 \leq k \leq n $ and $\dim{\orgr} = \dim{\R\widetilde{G}_{m,l}}$. Then $\deg{g} = \deg{h} = 0 $.
\end{thm}

\begin{proof}
Note that when $l=1$,  $\R \widetilde{G}_{m, l} \simeq S^{m-1}$. Hence there exists a map $h_{\lambda}  \colon \orgr \rightarrow \R \widetilde{G}_{m , 1}$ of any degree $\lambda \in \Z$ whenever $\dim(\orgr)=m-1$. Similarly, we have a map $g_\lambda \colon \R \widetilde{G}_{m,l} \rightarrow \widetilde{I}_{2n,1} $  of any specified degree $\lambda$ whenever $\dim(\R \widetilde{G}_{m,l})=2n-1$. 
 
If $ n = k =2$, $\dim{\orgr} = \dim{\R\widetilde{G}_{m,l}} = 3$ implies either $l = 1$ or $m-l = 1$ . Since $\R\widetilde{G}_{m,l} $ is diffeomorphic to $\R \widetilde{G}_{m, m-l}$, both these cases reduce to the cases discussed in the previous paragraph. 

Now consider the case where $n = k > 2$. Then, by  Proposition $\ref{initial}$,  we have $H^{4} (\widetilde{I}_{2n,n} ) = 0$ and $\pi_{1} (\widetilde{I}_{2n,n} ) = 0$, which respectively imply $\deg{h} = 0$ and $\deg{g} = 0$. 

Henceforth we restrict ourselves to the cases $2 \leq k < n $ and $2 \leq l < m $. Consider $h^{*} \colon H^{*}( \R\widetilde{G}_{m,l} ; \Q) \rightarrow H^{*}(\orgr ; \Q)$. Recall that $H^{4}(\R \widetilde{G}_{m,l} ; \Q)$ is generated by $p_{1}$ which has order $[l/2][(m-l)/2]$.  By Proposition $\ref{orderp1}$, order of $ p_{1} \in H^{4}(\orgr ; \Q)$ is  $[k/2][(n-k)/2]$ . And, $h^{*}$ takes $p_{1} \in H^{4}( \R \widetilde{G}_{m,l} ; \Q)$ to some  multiple of $p_{1}$ in $H^{4}(\orgr ; \Q)$. 

Therefore, as in the proof of Theorem $\ref{main1}$, we have that if $\deg{h} \neq 0$, $l(m-l)-4\leq k(n-k) \leq l(m-l) + 4$. Observe that $\dim{\orgr} = \dim{\R \widetilde{G}_{m,l}} $ implies $2k(n-k) + k(k+1)/2 = l(m-l)$. Hence the bound gives us $k(n-k)+k(k+1)/2 \leq 4$  which is not possible if $k\geq 2$ and $n>k$.  Therefore we have $\deg{h} = 0$. The proof that $\deg{g} = 0 $ is similar.
\end{proof}
\mbox{ }\\

 \begin{ack}
 The authors would like to thank P. Sankaran for suggesting the problem and for his helpful comments. The second author was partially supported by a grant from the J.C. Bose Fellowship of A. Bose. She would like to thank A. Bose for his support and encouragement. 
\end{ack}

\providecommand{\bysame}{\leavevmode\hbox to3em{\hrulefill}\thinspace}
\providecommand{\MR}{\relax\ifhmode\unskip\space\fi MR }
\providecommand{\MRhref}[2]{%
  \href{http://www.ams.org/mathscinet-getitem?mr=#1}{#2}
}
\providecommand{\href}[2]{#2}

\mbox{ } \\


\begin{thebibliography}{99}


\bibitem{mik} M.~Mikosz, {\em Secondary characteristic classes for the isotropic Grassmannian}, Geometry and Topology of Caustics, Banach Center Publ. 50, Warsaw, 1999, 195--204.


\bibitem{mt} M. Mimura and H. Toda, {\it Topology of Lie groups. I, II.} Translated from the 1978 Japanese edition by the authors. Translations of Mathematical Monographs, 91. American Mathematical Society, Providence, RI, 1991.

\bibitem{mor-nig} J.~Morvan,~L.~Niglio, {\em Isotropic characteristic classes}, Compo. Math. 91 (1994) 67--89.


\bibitem{psr} K.~H.~Paranjpe,~V.~Srinivas, {\em Self maps of homogeneous spaces}, Invent. Math. 98 (1989) 425--444. 

\bibitem{rs} V. Ramani and P. Sankaran, On degrees of maps between Grassmannians, Proc. Indian Acad. Sci. Math. Sci. {\bf 107} (1997), no.1, 13--19. 

\bibitem{ss} P.~Sankaran,~S.~Sarkar, {\em Degrees of maps between Grassmann manifolds}, Osaka J. Math. Volume 46, Number 4 (2009), 1143--1161.

 




\end{thebibliography}
\end{document}